\documentclass[oneside, a4paper]{article}
\usepackage[hidelinks]{hyperref}
\usepackage{amsmath, amsthm, amssymb, amsfonts}
\usepackage[left=2.5cm,right=2.5cm,top=3cm,bottom=3cm,a4paper]{geometry}
\usepackage{thmtools}
\usepackage{mathtools}
\usepackage{xcolor}
\newtheorem*{theorem}{Theorem}
\usepackage{array}
\makeatletter
\newcommand{\thickhline}{%
    \noalign {\ifnum 0=`}\fi \hrule height 1pt
    \futurelet \reserved@a \@xhline
}
\newcolumntype{"}{@{\hskip\tabcolsep\vrule width 1pt\hskip\tabcolsep}}
\makeatother

\title{\textbf{\Large{Arithmetic functions commutable with sums of squares}}}
\author{\large{Jungin Lee}}
\date{}

\usepackage{fancyhdr}

\newcommand\shorttitle{Arithmetic functions commutable with sums of squares}
\newcommand\authors{Jungin Lee}

\begin{document}
\maketitle

\vspace{0mm}

\textbf{Abstract.} In this note, we characterize all functions $f : \mathbb{N} \rightarrow \mathbb{C}$ such that $f(x_1^2+ \cdots + x_k^2)=f(x_1)^2+ \cdots + f(x_k)^2$, where $k \geq 3$ and $x_1, \cdots, x_k$ are positive integers. 

\vspace{5mm}

\section{Introduction}

In 1996, Chung \cite{chu96} classified multiplicative functions satisfying $f(m^2+n^2)=f(m)^2+f(n)^2$ for all $m, n \in \mathbb{N}$. Ba$\check{s}$i$\acute{c}$ \cite{bas14} generalized this result to arbitrary arithmetic functions. Recently, Park (\cite{park15}, \cite{park16}) proved that for every integer $k \geq 3$, a multiplicative function $f$ satisfies $f(x_1^2+ \cdots + x_k^2)=f(x_1)^2+ \cdots + f(x_k)^2$ for all $x_1, \cdots, x_k \in \mathbb{N}$ is an identity function. We generalize Park's result to arithmetic functions, as Ba$\check{s}$i$\acute{c}$ generalized Chung's result. \\

\section{Results} \label{section2}

\begin{theorem}
Let $k \geq 3$ be an integer. If a function $f : \mathbb{N} \rightarrow \mathbb{C}$ satisfies $f(x_1^2+ \cdots + x_k^2)=f(x_1)^2+ \cdots + f(x_k)^2$ for every $x_1, \cdots, x_k \in \mathbb{N}$, then one of the following holds: \\
(1) $f \equiv 0$ \\
(2) $f(n)=\left\{\begin{matrix}
n \,\, (\exists x_1, \cdots, x_k \in \mathbb{N} \,\, \,n=x_1^2+ \cdots + x_k^2) \\
\pm n \,\, (otherwise) \;\;\;\;\;\;\;\;\;\;\;\;\;\;\;\;\;\;\;\;\;\;\;\;\;\;\;\;\;\;\;\;\;\;\;\;\;
\end{matrix}\right.$ \\
(3) $f(n)=\left\{\begin{matrix}
\frac{1}{k} \,\, (\exists x_1, \cdots, x_k \in \mathbb{N} \,\, \,n=x_1^2+ \cdots + x_k^2) \\
\pm \frac{1}{k} \,\, (otherwise) \;\;\;\;\;\;\;\;\;\;\;\;\;\;\;\;\;\;\;\;\;\;\;\;\;\;\;\;\;\;\;\;\;\;\;\;\;
\end{matrix}\right.$ \\
(In (2) and (3), each $\pm$ sign is independent.)
\end{theorem}
\begin{proof}
Denote $f(n)^2$ by $A_n$. If $k=3$, then
\begin{equation} \label{eq1}
A_{n+2}+A_{n-2}+A_1=f(2n^2+9)=2A_n+A_3 \:\, (n \geq 3)
\end{equation}
so it is enough to show that (1), (2) or (3) holds for $n \leq 4$. If $k \geq 4$, then
\begin{equation} \label{eq2}
A_{n+1}+A_{n-1}+2A_2+(k-4)A_1=f(2n^2+k+6)=2A_n+(k-3)A_1+A_3 \:\, (n \geq 3)
\end{equation}
 so it is enough to show that (1), (2) or (3) holds for $n \leq 3$. For convenience, we denote $A_1, \cdots, A_5$ by $A, B, C, D$ and $E$. \\

\textbf{Case I.} $k=3$ 
\begin{subequations} \label{eq3}
\begin{gather}
C=(3f(1)^2)^2=9A^2 \label{eq3a} \\
C+2D=f(3)^2+2f(4)^2=f(41)=f(6)^2+f(1)^2+f(2)^2=(2A+B)^2+A+B \label{eq3b} \\
2A+(3B)^2=2f(1)^2+f(12)^2=f(146)=f(11)^2+f(3)^2+f(4)^2=(2A+C)^2+C+D \label{eq3c} \\
\begin{multlined}
(A+2C)^2+C+A=f(19)^2+f(3)^2+f(1)^2=f(371) \\
=f(17)^2+f(9)^2+f(1)^2=(2B+C)^2+(A+2B)^2+A 
\end{multlined} \label{eq3d} \\
E+2A=f(5)^2+2f(1)^2=f(27)=3f(3)^2=3C \label{eq3e} \\
E+2B=f(5)^2+2F(2)^2=f(33)=f(1)^2+2f(4)^2=A+2D  \label{eq3f} \\
\begin{multlined}
(2A+C)^2+A+B=f(11)^2+f(1)^2+f(2)^2=f(126) \\
=f(9)^2+f(6)^2+f(3)^2 =(A+2B)^2+(2A+B)^2+C  
\end{multlined} \label{eq3g}
\end{gather}
\end{subequations} \\
\noindent From the equations \eqref{eq3a}, \eqref{eq3b}, \eqref{eq3e} and \eqref{eq3f}, we can obtain $(B-4A)(B+8A-1)=0$. \\

(i) $B=4A$ : If we substitute $B=4A$ and $C=9A^2$ to the equation \eqref{eq3d}, we obtain $27A^2(A-1)(9A+5)=0$. If $A=0$, $B=C=0$ and $D=0$ by \eqref{eq3b}. If $A=1$, $B=4, C=9$ and $D=16$ by \eqref{eq3b}. If $A=-\frac{5}{9}$, $B=-\frac{20}{9}, C=\frac{25}{9}$ and this contradicts to \eqref{eq3g}. 

(ii) $B=1-8A$ : If we substitute $B=1-8A$ and $C=9A^2$ to the equations \eqref{eq3d} and \eqref{eq3g}, we obtain $(9A-1)(27A^3+39A^2-52A+8)=0$ and $(9A-1)(9A^3+5A^2-29A+4)=0$, respectively. Thus $A=\frac{1}{9}$, $B=C=\frac{1}{9}$, and $D=\frac{1}{9}$ by \eqref{eq3b}. \\

\noindent \textbf{Case II.} $k=4$ 
\begin{subequations} \label{eq4}
\begin{gather}
\begin{multlined}
(A+3B)^2+(2A+2B)^2+A+B=f(13)^2+f(10)^2+f(1)^2+f(2)^2 \\
=f(274)=f(16)^2+f(4)^2+2f(1)^2=(4B)^2+(4A)^2+2A
\end{multlined} \label{eq4a} \\
\begin{multlined}
(4B)^2+3B=f(16)^2+3f(2)^2=f(268) \\
=f(13)^2+2f(7)^2+f(1)^2=(A+3B)^2+2(3A+B)^2+A
\end{multlined} \label{eq4b} \\
(3A+B)^2+3A=f(7)^2+3f(1)^2=f(52)=3f(4)^2+f(2)^2=48A^2+B \label{eq4c} \\
E+3A=f(5)^2+3f(1)^2=f(28)=3f(3)^2+f(1)^2=3C+A \label{eq4d} \\
E+3B=f(5)^2+3f(2)^2=f(37)=2f(4)^2+f(1)^2+f(2)^2=32A^2+A+B \label{eq4e}
\end{gather}
\end{subequations} \\
\noindent From the equation \eqref{eq4a}, we can obtain $(A-B)(11A-3B+1)=0$. \\

(i) $B=A$ : If we substitute $B=A$ to the equation \eqref{eq4b}, we obtain $32A^2=2A$, so $A$ is $0$ or $\frac{1}{16}$. By the equations \eqref{eq4d} and \eqref{eq4e}, $(A,B,C)=(0,0,0), (\frac{1}{16}, \frac{1}{16}, \frac{1}{16})$. 

(ii) $B=\frac{11A+1}{3}$ : If we substitute $B=\frac{11A+1}{3}$ to the equations \eqref{eq4b} and \eqref{eq4c}, we obtain $(A-1)(160A+14)=0$ and $(A-1)(16A-1)=0$, respectively. Thus $A=1$, $B=4$ and $C=9$ by \eqref{eq4d} and \eqref{eq4e}. \\

\noindent \textbf{Case III.} $k \geq 5$ 
\begin{subequations} \label{eq5}
\begin{gather}
2(A-B)^2+A+2B=3C \label{eq5a} \\
5A+3C=8B \label{eq5b} 
\end{gather}
\end{subequations} \\
\noindent From $(kA)^2+((k-2)A+2B)^2+2B+(k-4)A=f(k)^2+f(k+6)^2+2f(2)^2+(k-4)f(1)^2=f(2k^2+13k+40)=2f(k+3)^2+3f(3)^2+(k-5)f(1)^2=2((k-1)A+B)^2+3C+(k-5)A$, we obtain \eqref{eq5a}. From $(D+4A)+(A+3C)=f(20)+f(28)=5B+(D+3B)$, we obtain \eqref{eq5b}. From the equations \eqref{eq5a} and \eqref{eq5b}, we obtain $2(A-B)(A-B+3)=0$. \\

(i) $B=A$ : By the equation \eqref{eq5b}, $C=A$ and by the equation \eqref{eq2}, $A_n=A$ for every $n \in \mathbb{N}$. Then $A=A_k=(kf(1)^2)^2=k^2A^2$, so $A$ is $0$ or $\frac{1}{k^2}$.  

(ii) $B=A+3$ : By the equation \eqref{eq5b}, $C=A+8$ and by the equation \eqref{eq2}, $A_n-n^2=A-1$ for every $n \in \mathbb{N}$. Then $((k-1)A+(A+3))^2-(k+3)^2=A_{k+3}-(k+3)^2=A_k-k^2=(kA)^2-k^2$, so $A=1$. 
\end{proof}

\vspace{0.5cm}

Department of Mathematics, Pohang University of Science and Technology, Pohang 790-784, Korea 

e-mail: moleculesum@postech.ac.kr

\end{document}